\documentclass[11pt,english,reqno]{amsart}

\usepackage{amsthm,amsmath,amssymb,amsfonts}
\usepackage{booktabs,float,caption,setspace}
\usepackage{mathtools,mleftright}
\usepackage[usenames]{color}
\usepackage{a4wide}

\definecolor{webgreen}{rgb}{0,.5,0}
\definecolor{webbrown}{rgb}{.6,0,0}

\usepackage[
  pdfauthor={},
  pdfkeywords={},
  pdftitle={},
  pdfcreator={},
  pdfproducer={},
  linktocpage,colorlinks,bookmarksnumbered,linkcolor=blue,
  citecolor=webgreen,urlcolor=webbrown]{hyperref}

\theoremstyle{plain}
\newtheorem{theorem}{Theorem}
\newtheorem{corl}[theorem]{Corollary}
\newtheorem{prop}[theorem]{Proposition}
\newtheorem{lemma}[theorem]{Lemma}

\theoremstyle{definition}

\newtheorem{remark}[theorem]{Remark}

\numberwithin{equation}{section}
\numberwithin{theorem}{section}
\numberwithin{figure}{section}
\numberwithin{table}{section}

\newcommand{\NN}{\mathbb{N}}
\newcommand{\ZZ}{\mathbb{Z}}
\newcommand{\QQ}{\mathbb{Q}}

\newcommand{\CC}{\mathbb{C}}
\newcommand{\DD}{\mathbb{D}}

\DeclareMathOperator{\ord}{ord}
\DeclareMathOperator{\real}{Re}
\DeclareMathOperator{\denom}{denom}
\DeclarePairedDelimiter{\norm}{|}{|}

\newcommand{\set}[1]{\left\{#1\right\}}
\newcommand{\mids}{\,\mid\,}
\newcommand{\nmids}{\,\nmid\,}
\newcommand{\valueat}[1]{{\,}_{\big|\, #1}}
\newcommand{\opdf}[2]{\nabla_{\hspace{-0.5ex}#1}^{\hspace{-0.03ex}#2}}
\newcommand{\opsh}{\eth}
\newcommand{\andq}{\quad \text{and} \quad}
\newcommand{\textq}[1]{\quad \text{#1} \quad}

\newcommand{\bb}{\mathbf{b}}
\newcommand{\BN}{\mathbf{B}}
\newcommand{\BNP}{\widehat{\BN}}
\newcommand{\BNPD}{\widehat{\overline{\BN}}}
\newcommand{\BNT}{\widetilde{\BN}}
\newcommand{\BNN}{\overline{\mathcal{B}}}
\newcommand{\QT}{\widetilde{Q}}
\newcommand{\SH}{\widehat{S}}
\newcommand{\WQ}{\mathcal{W}}

\newcommand{\arxiv}[1]{\href{https://arxiv.org/abs/#1}{arXiv:#1}}

\makeatletter
\newcommand{\spod}[1]{\allowbreak\if@display\mkern8mu\else\mkern4mu\fi(#1)}
\newcommand{\smod}[1]{\spod{{\operator@font mod}\mkern6mu#1}}
\makeatother

\begin{document}

\title[Wilson's theorem modulo higher prime powers II]
{Wilson's theorem modulo higher prime powers II:\\Bernoulli numbers and polynomials}
\author{Bernd C. Kellner}
\address{G\"ottingen, Germany}
\email{bk@bernoulli.org}

\begin{abstract}
By recent work of the author, Wilson's theorem as well as the Wilson quotient can 
be described by supercongruences of power sums of Fermat quotients modulo every 
higher prime power. We translate these congruences into congruences of power sums 
and Bernoulli numbers. This together provides relatively short proofs of the 
congruences compared to former approaches. As an application, we compute, e.g., 
the Wilson quotient up to modulo~$p^4$ and equivalently the factorial $(p-1)!$ 
up to modulo~$p^5$, which can be extended to any higher prime power with some effort.
As a by-product, we determine some power sums of the Fermat quotients up to modulo~$p^4$.
\end{abstract}

\maketitle


\section{Introduction}

By the well-known theorem of Wilson, one has
\[
  (p-1)! \equiv -1 \smod{p},
\]
whenever $p$ is a prime. In this context, the Wilson quotient is defined to be
\begin{equation} \label{eq:wp-def}
  \WQ_p = \frac{(p-1)! + 1}{p}.
\end{equation}
Let $p$ always denote a prime further on. 

In 1900, Glaisher~\cite{Glaisher:1900} and later Beeger~\cite{Beeger:1913} showed that
\[
 \WQ_p \equiv \BN_{p-1} + \tfrac{1}{p} - 1 \smod{p}
\]
in terms of the Bernoulli numbers $\BN_n$, which will be introduced later.
In~1938, E.~Lehmer~\cite{Lehmer:1938} derived more generally for $r \geq 1$ that
\begin{align*}
  r \WQ_p &\equiv \BN_{r(p-1)} + \tfrac{1}{p} - 1 \smod{p}, \\
\intertext{implying by difference that}
  \WQ_p &\equiv \BN_{2(p-1)} - \BN_{p-1} \smod{p}.
\end{align*}
Carlitz~\cite{Carlitz:1953} improved the result in 1953, as follows.
For $k \geq 0$ and $r \geq 1$, it holds that
\begin{equation} \label{eq:wp-car}
  r \WQ_p \equiv \frac{\BN_{r p^k (p-1)} + \tfrac{1}{p} - 1}{p^k} \smod{p}.
\end{equation}

Fermat's little theorem states that the congruence
\[
  a^{p-1} \equiv 1 \smod{p}
\]
holds for any integer $a$ coprime to $p$. Similarly, the Fermat quotient is given by
\begin{equation} \label{eq:qp-def}
  q_p(a) = \frac{a^{p-1} - 1}{p}.
\end{equation}
More generally, we consider sums of powers of the Fermat quotients defined by
\[
  Q_p(n) = \sum_{a=1}^{p-1} q_p(a)^n \quad (n \geq 1).
\]
Using this notation, Lerch~\cite{Lerch:1905} established the basic relationship in 1905 that
\[
  \WQ_p \equiv Q_p(1) \smod{p}.
\]

Recent work of the author revealed the general case, as follows.
\begin{theorem}[Kellner \cite{Kellner:2025}] \label{thm:kel}
We have the following statements:
\begin{enumerate}
\item There exist unique multivariate polynomials 
\[
  \psi_\nu(x_1,\ldots,x_\nu) \in \ZZ[x_1,\ldots,x_\nu] \quad (\nu \geq 1),
\]
which have no constant term and can be computed recursively; 

\item Let ${r \geq 1}$ and ${p > r}$ be an odd prime. Then we have 
\begin{align*}
  \WQ_p &\equiv \sum_{\nu=1}^{r} \frac{p^{\nu-1}}{\nu!} \, \psi_\nu(Q_p(1),\ldots,Q_p(\nu)) \smod{p^r} \\
\shortintertext{and equivalently}
  (p-1)! &\equiv -1 + \sum_{\nu=1}^{r} \frac{p^\nu}{\nu!} \, \psi_\nu(Q_p(1),\ldots,Q_p(\nu)) \smod{p^{r+1}}.
\end{align*}
\end{enumerate}
\end{theorem}

The polynomials $\psi_\nu$ can be calculated effectively with the help of 
Bell polynomials by a recurrence formula, see \cite{Kellner:2025}.
Table~\ref{tab:psi} below shows the first few polynomials~$\psi_\nu$ as
computed in \cite{Kellner:2025}.

\begin{table}[H] \small
\setstretch{1.25}
\begin{center}
\begin{tabular}{r@{\;=\;}l}
  \toprule
  $\psi_1$ & $x_1$ \\
  $\psi_2$ & $2 x_1 - x_1^2 - x_2$ \\
  $\psi_3$ & $6 x_1 - 6 x_1^2 + x_1^3 + 3 x_1 x_2 - 3 x_2 + 2 x_3$ \\
  $\psi_4$ & $24 x_1 - 36 x_1^2 + 12 x_1^3 - x_1^4 - 6 x_1^2 x_2 + 24 x_1 x_2 - 8 x_1 x_3 - 12 x_2 - 3 x_2^2 + 8 x_3 - 6 x_4$ \\
  \bottomrule
\end{tabular}

\caption{First few polynomials $\psi_\nu$.} \label{tab:psi}
\end{center}
\end{table}

We proceed according to the schema below in order to derive congruences of the 
Wilson quotient $\WQ_p$ for any prime power in terms of the Bernoulli numbers. 
\begin{figure}[H]
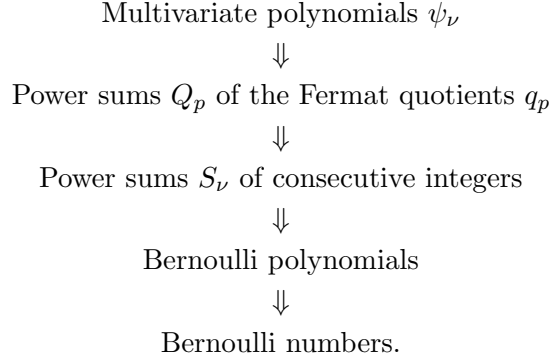

\begin{center}
\setstretch{1.2}
\begin{tabular}{c}
  Multivariate polynomials $\psi_\nu$ \\
  $\Downarrow$ \\
  Power sums $Q_p$ of the Fermat quotients $q_p$ \\
  $\Downarrow$ \\
  Power sums $S_\nu$ of consecutive integers \\
  $\Downarrow$ \\
  Bernoulli polynomials \\
  $\Downarrow$ \\
  Bernoulli numbers.
\end{tabular}
\end{center}
\caption{Schema of dependencies.} \label{fig:scheme}
\end{figure}

For $n \geq 1$ and any prime $p \geq 5$, define the divided Bernoulli numbers 
\begin{equation} \label{eq:bnn-1}
  \BNN_n = \frac{\BN_{n(p-1)}+\frac{1}{p}-1}{n(p-1)}, 
\end{equation}
which are $p$-integral and satisfy the Kummer congruences (following from \eqref{eq:wp-car})
\[
  \BNN_n \equiv \BNN_m \smod{p} \quad (n, m \geq 1).
\]
For $n=1$ and primes $p=2,3$, we extend the definition \eqref{eq:bnn-1} for $\BNN_1$.
Similarly, define
\begin{equation} \label{eq:bnn-2}
  \BNN_{n,2} = \frac{\BN_{n(p-1)-2}}{n(p-1)-2},
\end{equation}
which are also $p$-integral and obey the Kummer congruences. Throughout the paper, 
an \textsl{overbar}, e.g., $\BNN_n$, indicates a \textsl{divided} Bernoulli number.

The main results of the paper establish supercongruences of the Wilson quotient $\WQ_p$,
which can be extended to any higher prime power along the same methods.

\begin{theorem} \label{thm:main}
Let $p$ be a prime. Then we have
\begin{align*}
  \WQ_p &\equiv -\BNN_1 \smod{p}. \\
\shortintertext{For $p \geq 5$, we have}
  \WQ_p &\equiv -2 \BNN_1 + \BNN_2 - \tfrac{1}{2} p \, \BNN_1^2 \smod{p^2}, \\
  \WQ_p &\equiv -3 \BNN_1 + 3 \BNN_2 - \BNN_3
  + p \, \mleft( -\tfrac{3}{2} \BNN_1^2 + \BNN_1 \BNN_2 \mright)
  - p^2 \mleft( \tfrac{1}{6} \BNN_1^3 + \tfrac{1}{3} \BNN_{1,2} \mright) \smod{p^3}. 
\end{align*}
Equivalently, we have
\[
  (p-1)! \equiv -1 + p \WQ_p \smod{p^{r+1}}
\]
for $p$, $r \in \set{1,2,3}$, and $\WQ_p$ as above, respectively.
\end{theorem}

We additionally calculate the next case, which is even longer.

\begin{theorem} \label{thm:main2}
Let $p \geq 7$ be a prime. Then we have
\begin{align*}
  \WQ_p &\equiv -4 \BNN_1 + 6 \BNN_2 - 4 \BNN_3 + \BNN_4
  + p \, \mleft( -3 \BNN_1^2 + 4 \BNN_1 \BNN_2 - \BNN_1 \BNN_3 - \tfrac{1}{2} \BNN_2^2 \mright) \\
  &\quad + p^2 \mleft( -\tfrac{2}{3} \BNN_1^3 + \tfrac{1}{2} \BNN_1^2 \BNN_2
  - \tfrac{2}{3} \BNN_{1,2} + \tfrac{1}{3} \BNN_{2,2} \mright)
  - p^3 \mleft( \tfrac{1}{24} \BNN_1^4 + \tfrac{1}{3} \BNN_1 \BNN_{1,2} \mright) \smod{p^4}. 
\end{align*}
Equivalently, we have
\[
  (p-1)! \equiv -1 + p \WQ_p \smod{p^5}.
\]
\end{theorem}

\begin{remark} \label{rem:wp-red}
By the generalized Kummer congruences $\!\smod{p^r}$ (see Proposition~\ref{prop:gen-congr}), 
one can easily show the reduction of terms given by Theorems~\ref{thm:main} and \ref{thm:main2}
such that
\[
  \WQ_p \smod{p^4} \;\longmapsto\; \WQ_p \smod{p^3} \;\longmapsto\;
  \WQ_p \smod{p^2} \;\longmapsto\; \WQ_p \smod{p}.
\]
\end{remark}

As a by-product of the theory, we obtain supercongruences of the power sums $Q_p$ 
of the Fermat quotients $q_p$, as follows.

\begin{theorem} \label{thm:main3}
Let $p$ be an odd prime. Then we have 
\begin{alignat*}{3}
  Q_p(1) &\equiv (p-1) \, \BNN_1 - p^2 \, \BNN_{1,2} + \tfrac{11}{6} p^3 \, \BNN_{1,2}
  && \smod{p^4} &\quad& (p \geq 5) \\
  &\equiv (p-1) \, \BNN_1 - p^2 \, \BNN_{1,2}
  && \smod{p^3} && (p \geq 5) \\
  &\equiv (p-1) \, \BNN_1
  && \smod{p^2} && (p \geq 5) \\
  &\equiv -\BNN_1
  && \smod{p} && (p \geq 3), 
\end{alignat*}
\begin{alignat*}{3}
  \tfrac{1}{2} p \, Q_p(2) &\equiv (p-1) (\BNN_2 - \BNN_1)
  + p^2 \mleft( \BNN_{1,2} - 2 \BNN_{2,2} \mright) + \tfrac{5}{2} p^3 \, \BNN_{1,2}
  && \smod{p^4} &\quad& (p \geq 5) \\
  &\equiv (p-1) (\BNN_2 - \BNN_1) - p^2 \, \BNN_{1,2}
  && \smod{p^3} && (p \geq 5) \\
  &\equiv -(\BNN_2 - \BNN_1)
  && \smod{p^2} && (p \geq 3), 
\end{alignat*}
\begin{alignat*}{3}
  \tfrac{1}{3} p^2 \, Q_p(3) &\equiv (p-1) (\BNN_3 - 2 \BNN_2 + \BNN_1)
  + p^2 \mleft( \tfrac{7}{3} \BNN_{1,2} - \tfrac{8}{3} \BNN_{2,2} \mright) + p^3 \, \BNN_{1,2}
  && \smod{p^4} &\quad& (p \geq 7) \\
  &\equiv -(\BNN_3 - 2 \BNN_2 + \BNN_1) - \tfrac{1}{3} p^2 \, \BNN_{1,2}
  && \smod{p^3} && (p \geq 5), 
\end{alignat*}
and
\begin{alignat*}{3}
  \tfrac{1}{4} p^3 \, Q_p(4) &\equiv -(\BNN_4 - 3 \BNN_3 + 3 \BNN_2 - \BNN_1)
  + p^2 \mleft( \BNN_{1,2} - \BNN_{2,2} \mright)
  && \smod{p^4} &\quad& (p \geq 7). 
\end{alignat*}
\end{theorem}

The rest of the paper is organized as follows. The next two sections deal with the
Kummer congruences of the Bernoulli numbers, as well as with congruences of the 
power sums $Q_p$. Sections~\ref{sec:proof} and \ref{sec:proof2} contain the proofs 
of the main theorems, respectively. Section~\ref{sec:dis} concludes the paper with 
a discussion.


\section{Preliminaries}

Let $\NN = \set{1,2,3,\ldots}$ be the set of positive integers.
Let $\ZZ_p$ be the ring of $p$-adic integers, and $\QQ_p$ be the field of 
$p$-adic numbers. Define $\ord_p s$ as the $p$-adic valuation of $s \in \QQ_p$.
The ultrametric absolute value $\norm{\cdot}_p$ is defined by 
$\norm{s}_p = p^{-\ord_p s}$ on $\QQ_p$.

The Bernoulli polynomials $\BN_n(x)$ are defined by the generating function
\begin{equation} \label{eq:bnp-gf}
  \frac{t e^{xt}}{e^t - 1} = \sum_{n \geq 0} \BN_n(x) \frac{t^n}{n!} \quad (|t| < 2\pi).
\end{equation}
These polynomials are Appell polynomials (see \cite{Appell:1880}) 
and thus given by the formula
\begin{equation} \label{eq:bnp-def}
  \BN_n(x) = \sum_{k=0}^{n} \binom{n}{k} \BN_{n-k} \, x^k \quad (n \geq 0),
\end{equation}
where $\BN_n = \BN_n(0) \in \QQ$ is the $n$th Bernoulli number.
From the generating function \eqref{eq:bnp-gf}, it easily follows that $\BN_n$ 
vanishes for odd $n \geq 3$. First few numbers are $\BN_0 = 1$, $\BN_1 = -\frac{1}{2}$, 
and $\BN_2 = \frac{1}{6}$. Since $\BN_0 = 1$, all Bernoulli polynomials $\BN_n(x)$ 
are monic polynomials by \eqref{eq:bnp-def}. 

The divided Bernoulli numbers $\BN_n/n$ are $p$-integral whenever ${p-1 \nmid n}$ 
for even $n \geq 2$. The Kummer congruences, introduced by 
Kummer~\cite{Kummer:1851} in 1851, state that
\begin{equation} \label{eq:bn-congr}
  \frac{\BN_n}{n} \equiv \frac{\BN_m}{m} \smod{p},
\end{equation}
if $n \equiv m \not\equiv 0 \smod{p-1}$ and $n, m \in 2\NN$; implicitly requiring 
the $p$-integral property. Let 
\[
  \zeta(s) = \sum_{\nu \geq 1} \nu^{-s} \quad (s \in \CC, \; \real s > 1)
\] 
be the Riemann zeta function. Its functional equation leads to
\[
  \zeta(1-n) = -\frac{\BN_n}{n} \quad (n \geq 2).
\]

By the von~Staudt--Clausen theorem (announced by Clausen~\cite{Clausen:1840} 
without proof in 1840, while von~Staudt~\cite{Staudt:1840} then gave a rigorous 
proof at the same time), the Bernoulli numbers satisfy that
\[
  \BN_n + \sum_{p-1 \mids n} \frac{1}{p} \in \ZZ \quad (n \in 2\NN).
\]
As a consequence, we have $\BN_n + \tfrac{1}{p} \in \ZZ_p$ when $p-1 \mid n$.
By Carlitz~\cite{Carlitz:1953}, we define the $p$-integral Bernoulli numbers 
for odd prime $p$ and $n \in 2\NN_0$ by
\[
  \BNP_n =
  \begin{cases}
    0, & \text{if $n = 0$}; \\
    \BN_n + \tfrac{1}{p} - 1, & \text{if $n > 0$ and $p-1 \mid n$}; \\
    \BN_n, & \text{otherwise}.
  \end{cases}
\]

The divided $p$-integral Bernoulli numbers $\BNP_n/n$ have the following properties.
(Johnson~\cite{Johnson:1975} defined these numbers as $\beta_n$.)

\begin{prop} \label{prop:bnp-int}
Let $p \geq 5$ be a prime, ${n \in 2\NN}$, and ${n' \equiv n \smod{p-1}}$ with 
$0 \leq n' \leq p-3$. We have the following statements:
\begin{enumerate}
\item 
\[
  \ord_p \BNP_n \geq \ord_p n; 
\]
\item 
\[
  -\frac{\BNP_n}{n} \equiv
  \begin{cases}
    \WQ_p \hspace*{5.65ex} \smod{p}, & \text{if $n' = 0$}; \\
    \zeta(1-n') \smod{p}, & \text{otherwise}.
  \end{cases}
\]
\end{enumerate}
\end{prop}

\begin{proof}
If $n' \neq 0$, case~(1) follows from $\BN_n/n \in \ZZ_p$, while case~(2) follows 
from $\zeta(1-n') = -\BN_{n'}/n'$ and the Kummer congruences \eqref{eq:bn-congr}.
Now, assume $n' = 0$ and consider the decomposition $n = r p^k (p-1)$ with 
$k \geq 0$, $r \geq 1$, and $p \nmid r$. Then $\ord_p n = k$ and case~(1) is 
deduced from \eqref{eq:wp-car}. For case~(2), divide \eqref{eq:wp-car} by 
$r \in \ZZ_p^\times$ and multiply the right-hand side of \eqref{eq:wp-car} 
by $-1/(p-1)$, implying the result.
\end{proof}

With Carlitz's~\cite{Carlitz:1953} result, later reproved by Johnson~\cite{Johnson:1975},
we arrive at the unrestricted Kummer congruences, as follows.
 
\begin{corl}
Let $p \geq 5$ be a prime. For $n, m \in 2\NN$ with $n \equiv m \smod{p-1}$, 
we have
\[
  \frac{\BNP_n}{n} \equiv \frac{\BNP_m}{m} \smod{p}.
\]
\end{corl}

An odd prime $p$ is called a Wilson prime if $\WQ_p \equiv 0 \smod{p}$. 
So far, only the Wilson primes $5$, $13$, and $563$ have been found.
Any further Wilson prime must be greater than $2 \times 10^{13}$; 
see Costa et al.~\cite{CGH:2014}, also for the detailed history of 
the algorithms and search for these primes.

An odd prime $p$ is called irregular, if $p$ divides the numerator of any Bernoulli 
numbers $B_\ell$, equivalently to $B_\ell/\ell$, for even $\ell \in \set{2,4,\ldots,p-3}$; 
otherwise, $p$ is called regular. This classification of primes was introduced by 
Kummer~\cite{Kummer:1850} in 1850. More precisely, he proved that Fermat's last theorem 
is true in case the exponent is a regular prime. The only irregular primes below $100$ 
are $37$, $59$, and $67$. There exist infinitely many irregular primes,
see Carlitz~\cite{Carlitz:1954} for a short proof.

\begin{corl}
Let $n \in 2\NN$ and $p \geq 5$ be a prime. If $\norm{\BNP_n/n}_p < 1 $, then
$p$ is a Wilson prime, if $p-1 \mid n$; otherwise, $p$ is an irregular prime.
\end{corl}

The linear forward difference operator $\opdf{h}{}$ with step $h$
and its powers are defined by
\[
  \opdf{h}{n} f(s) = \sum_{\nu=0}^n \binom{n}{\nu} (-1)^{n-\nu} f(s+\nu h)
\]
for integers $n \geq 0$, $h \geq 1$, and any function $f: \ZZ_p \to \ZZ_p$. 
We use the expression, for example, $\opdf{h}{n} f(s+t) \valueat{s = 1}$
to indicate the variable and an initial value when needed.

The generalized Kummer congruences can be described, as follows.

\begin{prop} \label{prop:gen-congr}
Let $n \in 2\NN$, $p \geq 5$ be a prime, and $r \geq 1$. Then
\[
  \opdf{p-1}{r} \frac{\BNP_\nu}{\nu} \valueat{\nu = n} \equiv 0 \smod{p^r}
\]
holds for the following two mutually exclusive conditions:
\begin{enumerate}
\item $p-1 \nmid n$ and $n > r$;
\item $p-1 \mid n$ and $p > r + n/(p-1)$.
\end{enumerate}
\end{prop}

\begin{proof}
(1) See Kummer~\cite[p.\,371]{Kummer:1851}. 
(2) See Carlitz~\cite[Theorem\,1, p.\,166]{Carlitz:1953}.
\end{proof}

Define the power sums as usual by
\[
  S_n(m) = \sum_{\nu=1}^{m-1} \nu^n \quad (m, n \geq 0).
\]
Note that the case ${n=0}$ would normally include the summand for ${\nu = 0}$, 
but for practical reasons, we need here 
\[
  S_0(m) = m-1.
\]
It is then well known for $n \geq 1$ that
\begin{equation} \label{eq:sn-bn}
  S_n(x) = \int_{0}^{x} \BN_n(t) \, dt
  = \frac{1}{n+1}( \BN_{n+1}(x) - \BN_{n+1} ), 
\end{equation}
where $S_n(x)$ is an integer-valued polynomial of degree $n+1$ without constant term.

Define the linear $p$-adic backward shift operator by
\[
  \opsh_p: p \ZZ_p \to \ZZ_p, \quad
  a = \sum_{\nu \geq 1} a_\nu p^\nu \mapsto \frac{a}{p} = \sum_{\nu \geq 0} a_{\nu+1} p^\nu.
\]
This operator does not truncate a $p$-adic expansion, since it is defined on $p \ZZ_p$.
Finally, the power sums of the Fermat quotients can be expressed, as follows.

\begin{lemma} \label{lem:qp-diff}
For $n \geq 1$ and any prime $p$, we have 
\[
  Q_p(n) = \opsh_p^n \, \opdf{p-1}{n} \, S_\nu(p) \valueat{\nu = 0}.
\]
\end{lemma}

\begin{proof}
By \eqref{eq:qp-def}, summing 
\[
  q_p(a)^n = (\opsh_p (a^{p-1} - 1))^n
  = \opsh_p^n \, \opdf{p-1}{n} \, a^\nu \valueat{\nu = 0}
\]
over $a = 1,\ldots,p-1$ gives the result.
\end{proof}


\section{\texorpdfstring{$p$-adic}{p-adic} valuation of the Bernoulli polynomials}

In view of \eqref{eq:sn-bn}, define for $n \geq 1$ the related polynomials
\[
  \BNT_n(x) = \BN_n(x) - \BN_n,
\]
having no constant term. These polynomials have some interesting properties, as follows.
Define $s_p(n)$ as the sum of the base-$p$ digits of $n$.
Further define the product
\[
  \DD_n = \prod_{s_p(n) \, \geq \, p} p,
\]
which runs over the primes, see Kellner~\cite{Kellner:2017}. 
Since $s_p(n) = n$ for $p > n$, the product is always finite.
We have the remarkable relationship with the denominator of $\BNT_n(x)$ such that
\[
  \denom( \BNT_n(x) ) = \DD_n,
\]
which follows from the $p$-adic product formula. In particular, we have
\[
  \ord_p \BNT_n(x) =
  \begin{cases}
    -1, & \text{if $s_p(n) \geq p$}; \\
    \phantom{-} 0, & \text{otherwise}.
  \end{cases}
\]
As a further consequence, we have
\[
  \denom( S_n(x) ) = (n+1) \, \DD_{n+1}.
\]

Evaluating $S_n(x) = \BNT_{n+1}(x)/(n+1)$ with the help of \eqref{eq:bnp-def} 
provides the well-known formulas
\begin{align}
  S_n(x) &= \frac{1}{n+1} \sum_{k=1}^{n+1} \binom{n+1}{k} \BN_{n+1-k} \, x^k \label{eq:sn-bn-1} \\
  &= \sum_{\nu=0}^{n} \binom{n}{\nu} \BN_{n-\nu} \frac{x^{\nu+1}}{\nu+1}. \label{eq:sn-bn-2}
\end{align}
 
A \emph{folklore} result is the congruence for $n \in 2\NN$ and any prime $p$ that
\[
  S_n(p) \equiv p B_n \smod{p^r} \textq{where}
  \begin{cases}
    r = 2, & \text{if $p-1 \nmid n$}; \\
    r = 1, & \text{otherwise}.
  \end{cases}
\]
(Cf.~Ireland and Rosen \cite[Chap.\,15, pp.\,235--237]{IR:1990}.)
This can be achieved by a careful $p$-adic valuation of each term of the sum 
of the right-hand side of \eqref{eq:sn-bn-2}. Without knowledge of the factors of
$n+1$, \eqref{eq:sn-bn-1} is less suitable. However, in our case we have to consider
numbers of the form $n = d (p-1)$ with $d \geq 1$, and we can proceed, as follows.

Define the modified power sums by
\begin{equation} \label{eq:sh-def}
  \SH_n(x) = \frac{S_n(x) - S_0(x)}{x} \textq{where} \SH_0(p) = 0.
\end{equation}
To shorten the notation, let 
\begin{equation} \label{eq:bnpd-def}
  \BNPD_n = \frac{\BNP_n}{n} \quad (n \geq 1) \andq
  \BNPD_n = 0 \quad (n \leq 0),
\end{equation}
where the latter case is compatible with ${\SH_0(p) = 0}$.
We need a simple standard lemma (cf.~Robert~\cite{Robert:2000}).

\begin{lemma} \label{lem:log-ord}
Let $n \geq 1$ and $p$ be a prime. Then
\[
  \log_p n \geq \ord_p n,
\]
where $\log_p$ is the $\log$ function to base $p$.
\end{lemma}

\begin{lemma} \label{lem:binom-mod}
Let $d \geq 1$ and $p$ be a prime. Then
\[
  \ord_p \binom{d(p-1)}{p-1} =
  \begin{cases}
    0, & \text{if $d \equiv 1 \smod{p}$}; \\
    \delta \geq 1, & \text{otherwise}.
  \end{cases}
\]
\end{lemma}

\begin{proof}
We have $\binom{d(p-1)}{p-1} \equiv -(dp-d)_{p-1} \smod{p}$, where the falling 
factorial does not contain the factor $p$ if and only if $d \equiv 1 \smod{p}$
by a counting argument.
\end{proof}

\begin{prop} \label{prop:sh-eval}
Let $p \geq 5$ be a prime. Let $d \geq 1$ and $n = d(p-1)$. Then we have
\begin{equation} \label{eq:sh-eval}
  \SH_n(p) = \BNP_n + \sum_{\substack{\nu=2\\ 2 \mids \nu}}^{p-3}
  \binom{n}{\nu+1} \BNPD_{n-\nu} \, p^\nu +
  \begin{cases}
    \tfrac{1}{2} p^{p-2}, & \text{if $d = 1$}; \\
    O(p^{p-2}), & \text{if $d \not\equiv 1 \smod{p}$}; \\
    O(p^{p-3}), & \text{otherwise}.
  \end{cases}
\end{equation}
For the above coefficients, we have that $\BNP_n, \BNPD_{n-\nu} \in \ZZ_p$, 
and therefore $\SH_n(p) \in \ZZ_p$.
\end{prop}

\begin{proof}
Let $p \geq 5$ and $d=1$, so $n = p-1 \geq 4$. Note that $S_0(p)/p = 1 - \frac{1}{p}$. 
From \eqref{eq:sn-bn-2} and \eqref{eq:sh-def}, we infer that
\begin{align}
  \SH_n(p) &= \BNP_n + \sum_{\substack{\nu=2\\ 2 \mids \nu}}^{p-3}
  \binom{n}{\nu} \BN_{n-\nu} \frac{p^\nu}{\nu+1} + R(p) \label{eq:sh-rem} \\
  &= \BNP_n + \sum_{\substack{\nu=2\\ 2 \mids \nu}}^{p-3}
  \binom{n}{\nu+1} \BNPD_{n-\nu} \, p^\nu + R(p), \label{eq:sh-rem-2}
\end{align}
where the remainder term is given by
\[
  R(p) = \BN_1 \, p^{p-2} + \BN_0 \, p^{p-2} = \tfrac{1}{2} p^{p-2}.
\]
By definition, $\BNP_n$ is $p$-integral. This property also holds for the above 
coefficients $\BNPD_{n-\nu}$ by Proposition~\ref{prop:bnp-int}, since 
$n-\nu \equiv -\nu \not\equiv 0 \smod{p-1}$. Now, let $d \geq 2$ and $n = d(p-1)$. 
Then extending \eqref{eq:sh-rem} and \eqref{eq:sh-rem-2} yields
\[
  R(p) = R_1(p) + R_2(p) + R_3(p),
\]
where we split the summation and show that
\begin{align*}
  R_1(p) &= \sum_{\substack{\nu=p+1\\p-1 \nmids \nu\\ 2 \mids \nu}}^{n-2}
  \binom{n}{\nu+1} \BNPD_{n-\nu} \, p^\nu \in p^{p+1} \ZZ_p, \\
  R_2(p) &= \!\! \sum_{\substack{\ell=1\\ \nu=\ell(p-1)}}^{d-1} \hspace*{-1.5ex}
  \binom{n}{\nu} \BN_{n-\nu} \frac{p^\nu}{\nu+1} \in
  \begin{cases}
    p^{p-2} \ZZ_p, & \text{if $d \not\equiv 1 \smod{p}$}; \\
    p^{p-3} \ZZ_p, & \text{otherwise},
  \end{cases} \\
\shortintertext{and}
  R_3(p) &= - \frac{p^{n-1}}{2} + \frac{p^{n}}{n+1} \in p^{p-2} \ZZ_p.
\end{align*}

Case $R_1(p)$: The coefficients $\BNPD_{n-\nu}$ are $p$-integral, 
so the first term with ${\nu=p+1}$ $p$-adically wins,
giving $R_1(p) \in p^{p+1} \ZZ_p$.

Case $R_2(p)$: For ${1 \leq \ell < d}$ and ${\nu=\ell(p-1)}$, 
we have that ${\ord_p \BN_{n-\nu} = -1}$ by the von Staudt--Clausen theorem. 
Lemma~\ref{lem:log-ord} provides that
\begin{equation} \label{r2-estim-1}
  1 + \log_p \ell \geq 1 + \log_p \mleft( \ell \mleft( 1 - \tfrac{1}{p} \mright) + \tfrac{1}{p} \mright)
  = \log_p(\nu+1) \geq \ord_p(\nu+1).
\end{equation}
Define
\[
  L(\ell) = (p-1)(\ell-1) - \log_p \ell \andq \delta = \ord_p \binom{n}{\nu},
\]
where $L(1) = 0$ and $L(\ell) \geq 1$ for all $\ell \geq 2$. We then obtain
\begin{equation} \label{r2-estim-2}
\begin{aligned}
  \ord_p \mleft( \binom{n}{\nu} \BN_{n-\nu} \frac{p^\nu}{\nu+1} \mright)
  &= \delta - 1 + \nu - \ord_p(\nu+1) \\
  &\geq \delta + \ell(p-1) - 2 - \log_p \ell \\
  &= \delta + p-3 + L(\ell) \\
  &\geq \begin{cases}
    p-3, & \text{if $\ell = 1$ and $d \equiv 1 \smod{p}$}; \\
    p-2, & \text{otherwise}.
  \end{cases}
\end{aligned}
\end{equation}
The inequalities of \eqref{r2-estim-2} are sharp for $\ell = 1$ and $d \equiv 1 \smod{p}$ 
due to Lemma~\ref{lem:binom-mod}. This implies that $R_2(p) \in p^{p-2} \ZZ_p$ if $d \not\equiv 1 \smod{p}$,
and $R_2(p) \in p^{p-3} \ZZ_p$ otherwise.

Case $R_3(p)$: Similar to case $R_2(p)$, the inequalities of \eqref{r2-estim-1} 
and \eqref{r2-estim-2} with $\nu = d(p-1)$ imply that
\[
  \ord_p \frac{p^n}{n+1} \geq p-2.
\]
Since $\ord_p( p^{n-1}/2 ) = n-1 = d(p-1) - 1 > p-2$, this shows that $R_3(p) \in p^{p-2} \ZZ_p$.

Finally, putting all together, we derive in case $d \geq 2$ that 
$R(p) = O(p^{p-2})$ if $d \not\equiv 1 \smod{p}$, and $R(p) = O(p^{p-3})$ otherwise;
showing \eqref{eq:sh-eval}.
As a consequence, we conclude that $\SH_n(p) \in \ZZ_p$ for $d \geq 1$, completing the proof.
\end{proof}

\begin{corl} \label{cor:sh-cases}
Let $p \geq 5$ be a prime. Let $d \geq 1$ and $n = d(p-1)$. 
Define
\[
  \delta = \begin{cases}
    0, & \text{if $d \geq 2$ and $d \equiv 1 \smod{p}$}; \\
    1, & \text{otherwise}.
  \end{cases}
\]
In particular, if $d \leq p$, then $\delta = 1$. For $r \geq 1$ and $p \geq \max(5, r+3-\delta)$, we have
\begin{equation} \label{eq:sh-mod}
  \SH_n(p) \equiv \BNP_n + \sum_{\substack{\nu=2\\ 2 \mids \nu}}^{r-1}
  \binom{n}{\nu+1} \BNPD_{n-\nu} \, p^\nu \smod{p^r}.
\end{equation}
We obtain in case $d \leq p$ that
\[
  \SH_n(p) \equiv
  \begin{cases}
    \BNP_n \smod{p}, & \text{if $p \geq 5$}; \\
    \BNP_n \smod{p^2}, & \text{if $p \geq 5$}; \\
    \BNP_n + p^2 \binom{n}{3} \BNPD_{n-2} \smod{p^3}, & \text{if $p \geq 5$}; \\
    \BNP_n + p^2 \binom{n}{3} \BNPD_{n-2} \smod{p^4}, & \text{if $p \geq 7$}; \\
    \BNP_n + p^2 \binom{n}{3} \BNPD_{n-2} + p^4 \binom{n}{5} \BNPD_{n-4} \smod{p^5}, & \text{if $p \geq 7$}; \\
    \BNP_n + p^2 \binom{n}{3} \BNPD_{n-2} + p^4 \binom{n}{5} \BNPD_{n-4} \smod{p^6}, & \text{if $p \geq 11$}.
  \end{cases}
\]
\end{corl}

\begin{proof}
By Proposition~\ref{prop:sh-eval}, we have the expansion
\[
  \SH_n(p) = \alpha_0 + \alpha_2 \, p^2 + \cdots + \alpha_{p-3} \, p^{p-3} + O(p^{p-3+\delta})
\] 
with some coefficients $\alpha_\nu \in \ZZ_p$ depending on $n$ and $p$.
Note that $\alpha_{p-3}$ cannot be determined if $\delta = 0$ and the remainder term is $O(p^{p-3})$.
This gives the bound $r \leq p-3+\delta$, which yields $p \geq r+3-\delta$ and $p \geq 5$ by assumption,
implying the congruence \eqref{eq:sh-mod}.
\end{proof}

\begin{lemma} \label{lem:bin-diff}
For $k, n \geq 1$,  and any prime $p > k$, we have 
\[
  \opdf{p-1}{n} \binom{\nu}{k} \valueat{\nu = 0}
  \equiv (-1)^k \binom{k-1}{n-1} \smod{p}.
\]
\end{lemma}

\begin{proof}
Note that for $n > k$ both sides of the above congruence vanish, 
so let $1 \leq n \leq k$. Since $p > k$, we infer for $\nu \geq 0$ that
\[
  \binom{\nu(p-1)}{k} \equiv \binom{-\nu}{k}
  \equiv (-1)^k \binom{k-1+\nu}{k} \smod{p}.
\]
By rules of the difference operator $\opdf{}{}$, we then obtain
\[
  \opdf{p-1}{n} \binom{\nu}{k} \valueat{\nu = 0}
  \equiv \opdf{}{n} \binom{\nu(p-1)}{k} \valueat{\nu = 0}
  \equiv (-1)^k \binom{k-1}{k-n}
  \equiv (-1)^k \binom{k-1}{n-1} \smod{p}. \qedhere
\]  
\end{proof}

\begin{prop} \label{prop:qp-bnp}
For $n \geq 1$ and any prime $p$, we have 
\[
  Q_p(n) = \opsh_p^{n-1} \, \opdf{p-1}{n} \, \SH_\nu(p) \valueat{\nu = 0}.
\]
In particular, we have for $n \in \set{1,2,3}$ and $p \geq 5$ that
\begin{align}
  p^{n-1} Q_p(n) &\equiv \opdf{p-1}{n} \, \BNP_\nu \valueat{\nu = 0}
  - \alpha_n \, p^2 \, \BNPD_{p-3} \smod{p^3} \label{eq:bnp-p3} \\
\shortintertext{and equivalently}
  \frac{1}{n} p^{n-1} Q_p(n) &\equiv (p-1) \opdf{p-1}{n-1} \, \BNPD_\nu \valueat{\nu = p-1}
  - \frac{\alpha_n}{n} \, p^2 \, \BNPD_{p-3} \smod{p^3}, \label{eq:bnpd-p3}
\end{align}
where $\alpha = (1,2,1)$.
\end{prop}

\begin{proof}
Let $n \geq 1$ and $p$ be a prime. Since constant terms vanish in differences, 
it follows from Lemma~\ref{lem:qp-diff} and \eqref{eq:sh-def} that
\[
  Q_p(n) = \opsh_p^{n-1} \, \opdf{p-1}{n} \, \frac{S_\nu(p)}{p} \valueat{\nu = 0}
  = \opsh_p^{n-1} \, \opdf{p-1}{n} \, \SH_\nu(p) \valueat{\nu = 0}.
\] 

Now, let $n \in \set{1,2,3}$ and $p \geq 5$, so $n < p$. Remember that $\SH_0(p) = \BNP_0 = 0$
and $\BNPD_{\nu} = 0$ for ${\nu < 0}$, being compatible with $\SH_0(p) = 0$.
Using Corollary~\ref{cor:sh-cases} in the case$\smod{p^3}$, we infer that
\[
  p^{n-1} Q_p(n) \equiv \opdf{p-1}{n} \, \BNP_\nu \valueat{\nu = 0}
  + p^2 \, \opdf{p-1}{n} \binom{\nu}{3} \BNPD_{\nu-2} \valueat{\nu = 0} \smod{p^3}.
\]
Due to the Kummer congruences and $\binom{0}{3} = 0$, we obtain
\[
  \opdf{p-1}{n} \binom{\nu}{3} \BNPD_{\nu-2} \valueat{\nu = 0}
  \equiv \opdf{p-1}{n} \binom{\nu}{3} \BNPD_{p-3} \valueat{\nu = 0} \smod{p}
\]
and
\[
  \alpha'_n \equiv \opdf{p-1}{n} \binom{\nu}{3} \valueat{\nu = 0} \smod{p}
\]
with $\alpha' = -(1,2,1)$ by Lemma~\ref{lem:bin-diff}.
Putting all together and converting signs imply \eqref{eq:bnp-p3}.

For the step from \eqref{eq:bnp-p3} to \eqref{eq:bnpd-p3}, we need to show that
\[
  \opdf{p-1}{n} \, \BNP_\nu \valueat{\nu = 0}
  = n (p-1) \opdf{p-1}{n-1} \, \BNPD_\nu \valueat{\nu = p-1}.
\]
Since $\BNP_0 = 0$, this follows from
\[
  \sum_{\nu=1}^n \binom{n}{\nu} (-1)^{n-\nu} \, \BNP_{\nu(p-1)}
  = n(p-1) \sum_{\nu=1}^n \binom{n-1}{\nu-1} (-1)^{n-\nu} \, \BNPD_{\nu(p-1)},
\]
completing the proof.
\end{proof}

\begin{prop} \label{prop:qp-bnp-2}
Let $1 \leq n \leq 4$ and $p \geq 7$ be a prime. Then we have
\[
  \frac{1}{n} p^{n-1} Q_p(n) \equiv (p-1) \opdf{p-1}{n-1} \, \BNPD_\nu \valueat{\nu = p-1}
  + \frac{\alpha_n}{n} \, p^2 \, \BNPD_{p-3}
  + \frac{\beta_n}{n} \, p^2 \, \BNPD_{2p-4}
  + \frac{\gamma_n}{n} \, p^3 \, \BNPD_{p-3} \smod{p^4}, 
\]
where $\alpha = (-1,2,7,4)$, $\beta = -(0,4,8,4)$, and $\gamma = (\frac{11}{6},5,3,0)$.
\end{prop}

\begin{proof}
Note that $n < p$. We use Corollary~\ref{cor:sh-cases} in the case$\smod{p^4}$. 
Similar to the proof of Proposition~\ref{prop:qp-bnp}, we then obtain
\[
  \frac{1}{n} p^{n-1} Q_p(n) \equiv (p-1) \opdf{p-1}{n-1} \, \BNPD_\nu \valueat{\nu = p-1}
  + \frac{1}{n} p^2 \, \opdf{p-1}{n} \binom{\nu}{3} \BNPD_{\nu-2} \valueat{\nu = 0} \smod{p^4}.
\]
We further substitute $\bb_j = \BNPD_{j(p-1)-2}$ and evaluate that
\[
  \opdf{p-1}{n} \binom{\nu}{3} \BNPD_{\nu-2} \valueat{\nu = 0}
  \equiv \sum_{j=1}^{n} (s_{n,j} + t_{n,j} \, p) \, \bb_j \smod{p^2}
\]
with some integer coefficients $s_{n,j}$ and $t_{n,j}$. 
By the Kummer congruences, we determine that
\[
  \sum_{j=1}^{n} t_{n,j} \, \bb_j \equiv \gamma_n \, \bb_1 \smod{p}
\]
with $\gamma = (\frac{11}{6},5,3,0)$. 
By the generalized Kummer congruences, we have for all $j \geq 1$ that
\[
  \bb_j - 2 \bb_{j+1} + \bb_{j+2} \equiv 0 \smod{p^2}.
\]
For the coefficients $s_{n,j}$, we compute the following expressions and 
their reduction $\!\smod{p^2}$:
\begin{center} \small
\begin{tabular}{cl}
\toprule
  $n=1$ & $- \bb_1 \smod{p^2}$ \\
  $n=2$ & $2 \bb_1 - 4 \bb_2 \smod{p^2}$ \\
  $n=3$ & $-3 \bb_1 + 12 \bb_2 - 10 \bb_3 \equiv 7 \bb_1 - 8 \bb_2 \smod{p^2}$ \\
  $n=4$ & $4 \bb_1 - 24 \bb_2 + 40 \bb_3 - 20 \bb_4 \equiv 4 \bb_1 - 4 \bb_2 \smod{p^2}$. \\
\bottomrule
\end{tabular}
\end{center} 
This defines $\alpha = (-1,2,7,4)$ and $\beta = -(0,4,8,4)$, and completes the proof.
\end{proof}


\section{Proof of Theorem~\ref{thm:main}}
\label{sec:proof}

Remember the notation of \eqref{eq:bnn-1} and \eqref{eq:bnn-2}. 
From \eqref{eq:bnpd-def}, it follows for $n \geq 1$ and any prime $p \geq 5$ that
\[
  \BNN_n = \BNPD_{n(p-1)} \andq \BNN_{n,2} = \BNPD_{n(p-1)-2}.
\]
These numbers, lying in $\ZZ_p$, satisfy the generalized Kummer congruences of 
Proposition~\ref{prop:gen-congr}. The congruences~\eqref{eq:bnpd-p3} of 
Proposition~\ref{prop:qp-bnp} then turn easily into the following congruences.

\begin{lemma} \label{lem:qp-bnn}
For any prime $p \geq 5$, we have 
\begin{align*}
  Q_p(1) &\equiv (p-1) \, \BNN_1 - p^2 \, \BNN_{1,2} \smod{p^3} \\
  &\equiv (p-1) \, \BNN_1 \smod{p^2} \\
  &\equiv -\BNN_1 \smod{p}, \\
  \tfrac{1}{2} p \, Q_p(2) &\equiv (p-1) (\BNN_2 - \BNN_1) - p^2 \, \BNN_{1,2} \smod{p^3} \\
  &\equiv -(\BNN_2 - \BNN_1) \smod{p^2}, \\
  \tfrac{1}{3} p^2 \, Q_p(3) &\equiv -(\BNN_3 - 2 \BNN_2 + \BNN_1)
  - \tfrac{1}{3} p^2 \, \BNN_{1,2} \smod{p^3}.
\end{align*}
\end{lemma}

We split the proof of Theorem~\ref{thm:main} into three parts, as follows.
\begin{prop} \label{prop-p1}
Let $p$ be a prime. Then we have 
\[
  \WQ_p \equiv -\BNN_1 \smod{p}. 
\]
\end{prop}

\begin{proof}
Note that $\WQ_2 = \WQ_3 = 1$ and $\BNN_1 = -1, -\frac{1}{4}$ for $p=2,3$, 
respectively. For these cases, the above congruence holds. Now, let $p \geq 5$.
From Theorem~\ref{thm:kel}, Table~\ref{tab:psi}, and Lemma~\ref{lem:qp-bnn}, 
we infer that $\WQ_p \equiv Q_p(1) \equiv -\BNN_1 \smod{p}$.
\end{proof}

\begin{prop} \label{prop-p2}
Let $p \geq 5$ be a prime. Then we have 
\[
  \WQ_p \equiv -2 \BNN_1 + \BNN_2 - \tfrac{1}{2} p \, \BNN_1^2 \smod{p^2}.
\]
\end{prop}

\begin{proof}
From Theorem~\ref{thm:kel} and Table~\ref{tab:psi}, we deduce that
\[
  \WQ_p \equiv Q_p(1) + \tfrac{1}{2} p \left( 2 Q_p(1) - Q_p^2(1) - Q_p(2) \right) \smod{p^2}. 
\]
By Lemma~\ref{lem:qp-bnn}, this translates into
\begin{align*}
  \WQ_p &\equiv (p-1) \, \BNN_1
  - \tfrac{1}{2} p \, ( 2\BNN_1 + \BNN_1^2 )
  + (\BNN_2 - \BNN_1) \\
  &\equiv -2 \BNN_1 + \BNN_2 - \tfrac{1}{2} p \, \BNN_1^2 \smod{p^2}. \qedhere
\end{align*}
\end{proof}

\begin{prop} \label{prop-p3}
Let $p \geq 5$ be a prime. Then we have 
\begin{equation} \label{eq:wp-p3}
  \WQ_p \equiv -3 \BNN_1 + 3 \BNN_2 - \BNN_3
  - \tfrac{3}{2} p \, \BNN_1^2 + p \, \BNN_1 \BNN_2
  - \tfrac{1}{6} p^2 \, \BNN_1^3 - \tfrac{1}{3} p^2 \, \BNN_{1,2} \smod{p^3}. 
\end{equation}
\end{prop}

\begin{proof}
Theorem~\ref{thm:kel} and Table~\ref{tab:psi} provide that
\begin{align*}
  \WQ_p &\equiv Q_p(1) \\
  &\quad + \tfrac{1}{2} p \left( 2 Q_p(1) - Q_p^2(1) - Q_p(2) \right) \\
  &\quad + \tfrac{1}{6} p^2 \left( 6 Q_p(1) - 6 Q_p(1)^2 + Q_p(1)^3
  + 3 Q_p(1) Q_p(2) - 3 Q_p(2) + 2 Q_p(3) \right) \smod{p^3}. 
\end{align*}
Replacing terms by Lemma~\ref{lem:qp-bnn} yields
\begin{align*}
  \WQ_p &\equiv (p-1) \, \BNN_1 - p^2 \, \BNN_{1,2} \\
  &\quad + p(p-1) \, \BNN_1 - \tfrac{1}{2} p (p-1)^2 \, \BNN_1^2 - (p-1) (\BNN_2 - \BNN_1) + p^2 \, \BNN_{1,2} \\
  &\quad + \Bigl( p^2 ( -\BNN_1 - \BNN_1^2 - \tfrac{1}{6} \BNN_1^3 )
  + p (p-1)^2 (\BNN_2 - \BNN_1) \BNN_1 + p \, (\BNN_2 - \BNN_1) \\
  &\qquad \;\; - (\BNN_3 - 2 \BNN_2 + \BNN_1) - \tfrac{1}{3} p^2 \, \BNN_{1,2} \Bigr) \smod{p^3}.
\end{align*}
After expanding and rearranging of terms (e.g., using \textsl{Mathematica}), we obtain
\[
  \WQ_p \equiv \omega_1 + \omega_2 \smod{p^3},
\]
where $\omega_1$ equals the right-hand side of \eqref{eq:wp-p3}. 
The remaining terms are given by
\[
  \omega_2 \equiv 2 p^2 \, \BNN_1 (\BNN_1 - \BNN_2) \equiv 0 \smod{p^3},
\]
which vanish by the Kummer congruences. This completes the proof.
\end{proof}

\begin{proof}[Proof of Theorem~\ref{thm:main}]
This follows from Propositions \ref{prop-p1} -- \ref{prop-p3} along with \eqref{eq:wp-def}.
\end{proof}

\begin{remark} 
As noted by Remark~\ref{rem:wp-red}, we would only need Proposition~\ref{prop-p3} to derive
also the congruences $\WQ_p \smod{p^r}$ for $r \in \set{1,2}$ by reduction (we leave this
as an exercise for the reader). However, Propositions \ref{prop-p1} -- \ref{prop-p3} are 
given separatively to show the simple and short proof in each case. 
\end{remark}


\section{Proofs of Theorems~\ref{thm:main2} and \ref{thm:main3}}
\label{sec:proof2}

We extend Lemma~\ref{lem:qp-bnn} for the case $\!\smod{p^4}$, as follows.
\begin{lemma} \label{lem:qp-bnn-2}
Let $p \geq 7$ be a prime. Then we have
\begin{align*}
  Q_p(1) &\equiv (p-1) \BNN_1 - p^2 \, \BNN_{1,2} + \tfrac{11}{6} p^3 \, \BNN_{1,2} \smod{p^4}, \\
  \tfrac{1}{2} p \, Q_p(2) &\equiv (p-1) (\BNN_2 - \BNN_1)
  + p^2 \, \BNN_{1,2} - 2 p^2 \, \BNN_{2,2} + \tfrac{5}{2} p^3 \, \BNN_{1,2} \smod{p^4}, \\
  \tfrac{1}{3} p^2 \, Q_p(3) &\equiv (p-1) (\BNN_3 - 2 \BNN_2 + \BNN_1)
  + \tfrac{7}{3} p^2 \, \BNN_{1,2} - \tfrac{8}{3} p^2 \, \BNN_{2,2} + p^3 \, \BNN_{1,2} \smod{p^4}, \\
  \tfrac{1}{4} p^3 \, Q_p(4) &\equiv -(\BNN_4 - 3 \BNN_3 + 3 \BNN_2 - \BNN_1)
  + p^2 \, \BNN_{1,2} - p^2 \, \BNN_{2,2} \smod{p^4}.
\end{align*}
\end{lemma}

\begin{proof}
This follows from Proposition~\ref{prop:qp-bnp-2}. The congruence for $Q_p(4)$ 
can be reduced (i.e., the term $p-1$ is replaced by $-1$)
by the generalized Kummer congruences of Proposition~\ref{prop:gen-congr}.
\end{proof}

We can apply Lemma~\ref{lem:qp-bnn-2} together with Lemma~\ref{lem:qp-bnn} to substitute 
values of $Q_p$ in different moduli. We use \textsl{Mathematica} to shorten lengthy calculations.
The remaining terms are then simplified and reduced by applying the Kummer congruences afterwards.
We use the substitution 
\[
  \QT_p(n) = \frac{1}{n} p^{n-1} Q_p(n).
\]

\begin{proof}[Proof of Theorem~\ref{thm:main2}]
By Theorem~\ref{thm:kel} and Table~\ref{tab:psi}, we obtain
\[
  \WQ_p \equiv \omega_1 + \omega_2 + \omega_3 + \omega_4 \smod{p^4},
\]
where
\begin{align*}
  \omega_1 &\equiv Q_p(1) \equiv \QT_p(1) \smod{p^4}, 
\end{align*}
\begin{align*}
  \omega_2 &\equiv \tfrac{1}{2} p \mleft( 2 Q_p(1) - Q_p^2(1) - Q_p(2) \mright) \\
  &\equiv p \mleft( \QT_p(1) - \tfrac{1}{2} \QT_p^2(1) \mright) - \QT_p(2) \smod{p^4}, 
\end{align*}
\begin{align*}
  \omega_3 &\equiv \tfrac{1}{6} p^2 \mleft(
  6 Q_p(1) - 6 Q_p(1)^2 + Q_p(1)^3 + 3 Q_p(1) Q_p(2) - 3 Q_p(2) + 2 Q_p(3) \mright) \\
  &\equiv p^2 \mleft( \QT_p(1) - \QT_p(1)^2 + \tfrac{1}{6} \QT_p(1)^3 \mright)
  + p \mleft( \QT_p(2) (\QT_p(1)-1) \mright) + \QT_p(3) \smod{p^4}, 
\end{align*}
and
\begin{align*}
  \omega_4 &\equiv \tfrac{1}{24} p^3 \bigl(
  24 Q_p(1) - 36 Q_p(1)^2 + 12 Q_p(1)^3 - Q_p(1)^4 - 6 Q_p(1)^2 Q_p(2) \\
  &\quad \, + 24 Q_p(1) Q_p(2) - 8 Q_p(1) Q_p(3) - 12 Q_p(2) - 3 Q_p(2)^2 + 8 Q_p(3) - 6 Q_p(4) \bigr) \\
  &\equiv p^3 \mleft( \QT_p(1) - \tfrac{3}{2} \QT_p(1)^2 + \tfrac{1}{2} \QT_p(1)^3 - \tfrac{1}{24} \QT_p(1)^4 \mright) \\
  &\quad \, + p^2 \mleft( 2 \QT_p(1) \QT_p(2) - \tfrac{1}{2} \QT_p(1)^2 \, \QT_p(2) - \QT_p(2) \mright) \\
  &\quad \, + p \mleft( -\tfrac{1}{2} \QT_p(2)^2 - \QT_p(1) \QT_p(3) + \QT_p(3) \mright) - \QT_p(4) \smod{p^4}.
\end{align*}

The computation leads to
\[
  \WQ_p \equiv \omega'_1 + p \, \omega'_2 + p^2 \omega'_3 + p^3 \omega'_4 \smod{p^4}
\]
with the following terms and reduced ones by the Kummer congruences, respectively.
Namely,
\begin{align*}
  \omega'_1 &\equiv -4 \BNN_1 + 6 \BNN_2 - 4 \BNN_3 + \BNN_4 \smod{p^4}, 
\end{align*}
\begin{align*}
  \omega'_2 &\equiv -3 \BNN_1^2 + 4 \BNN_1 \BNN_2 - \BNN_1 \BNN_3 - \tfrac{1}{2} \BNN_2^2 \smod{p^3}, 
\end{align*}
\begin{align*}
  \omega'_3 &\equiv 2 \BNN_1^2 - \tfrac{2}{3} \BNN_1^3 + \tfrac{1}{2} \BNN_1^2 \BNN_2
  - 4 \BNN_1 \BNN_2 + \BNN_1 \BNN_3 + \BNN_2^2
  - \tfrac{2}{3} \BNN_{1,2} + \tfrac{1}{3} \BNN_{2,2} \\
  &\equiv (\BNN_1 - \BNN_2)^2 - \tfrac{2}{3} \BNN_1^3 + \tfrac{1}{2} \BNN_1^2 \BNN_2
  - \tfrac{2}{3} \BNN_{1,2} + \tfrac{1}{3} \BNN_{2,2} \\
  &\equiv -\tfrac{2}{3} \BNN_1^3 + \tfrac{1}{2} \BNN_1^2 \BNN_2
  - \tfrac{2}{3} \BNN_{1,2} + \tfrac{1}{3} \BNN_{2,2} \smod{p^2}, 
\end{align*}
and
\begin{align*}
  \omega'_4 &\equiv \tfrac{1}{2} \BNN_1^2 + \BNN_1^3 - \tfrac{1}{24} \BNN_1^4
  - \BNN_1^2 \BNN_2 - \tfrac{1}{2} \BNN_2^2
  - \tfrac{4}{3} \BNN_1 \BNN_{1,2} + 2 \BNN_2 \BNN_{1,2} - \BNN_3 \BNN_{1,2} \\
  &\equiv -\tfrac{1}{24} \BNN_1^4 - \tfrac{1}{3} \BNN_1 \BNN_{1,2} \smod{p}.
\end{align*}
This shows the result.
\end{proof}

\begin{proof}[Proof of Theorem~\ref{thm:main3}]
This follows from Lemmas~\ref{lem:qp-bnn} and \ref{lem:qp-bnn-2}.
The formulas were checked whether they also hold for smaller primes $p=3, 5$. 
\end{proof}


\section{Discussion}
\label{sec:dis}

We compare the results of Theorems~\ref{thm:main} and \ref{thm:main2} with 
former approaches to evaluate ${(p-1)!} \smod{p^r}$ for ${r \geq 2}$, 
which use completely different methods. For example, this leads to congruences 
of convolution sums of the Bernoulli numbers. 

In 1900, Glaisher~\cite[p.\,325]{Glaisher:1900} obtained the congruence
\[
  (p-1)! \equiv -1 + p \mleft( -1 + \BN_{p-1} + \tfrac{1}{p} \mright) \smod{p^2},
\]
which is equivalent to
\[
  \WQ_p \equiv -\BNN_1 \smod{p}.
\]

100 years later, Z.\,H.\,Sun~\cite[p.\,195]{Sun:2000} derived the next result 
\[
  (p-1)! \equiv p \frac{\BN_{2p-2}}{2p-2} - p \frac{\BN_{p-1}}{p-1}
  - \frac{1}{2} \mleft( p \frac{\BN_{p-1}}{p-1} \mright)^{\! 2} \! \smod{p^3},
\]
which is equivalent to
\[
  \WQ_p \equiv -2 \BNN_1 + \BNN_2 - \tfrac{1}{2} p \, \BNN_1^2 \smod{p^2}. 
\]

Recently, Levaillant~\cite[pp.\,83,\,110]{Levaillant:2020} gave a result for 
$(p-1)! \smod{p^4}$ by extending work of Z.\,H.\,Sun~\cite{Sun:2000}. 
The formula is impractical to use due to numerous $p$-adic nested terms 
and requires about one page to be written out in full. The proofs are extremely 
complicated and very lengthy. Here we present the 7-term relation 
\[
  \WQ_p \equiv -3 \BNN_1 + 3 \BNN_2 - \BNN_3
  + p \, \mleft( -\tfrac{3}{2} \BNN_1^2 + \BNN_1 \BNN_2 \mright)
  - p^2 \mleft( \tfrac{1}{6} \BNN_1^3 + \tfrac{1}{3} \BNN_{1,2} \mright) \smod{p^3}. 
\]

The presented theory in this paper allows us to compute the 
Wilson quotient $\WQ_p$ and $(p-1)!$ modulo any higher prime power of $p$,
which is achieved by Theorem~\ref{thm:kel} in terms of the power sums $Q_p$
of the Fermat quotients $q_p$. The values of $Q_p$ can then be translated into 
congruences of the Bernoulli numbers as provided by Theorem~\ref{thm:main3}.

Even the next case is computable with little effort as
\begin{align*}
  \WQ_p &\equiv -4 \BNN_1 + 6 \BNN_2 - 4 \BNN_3 + \BNN_4
  + p \, \mleft( -3 \BNN_1^2 + 4 \BNN_1 \BNN_2 - \BNN_1 \BNN_3 - \tfrac{1}{2} \BNN_2^2 \mright) \\
  &\quad + p^2 \mleft( -\tfrac{2}{3} \BNN_1^3 + \tfrac{1}{2} \BNN_1^2 \BNN_2
  - \tfrac{2}{3} \BNN_{1,2} + \tfrac{1}{3} \BNN_{2,2} \mright)
  - p^3 \mleft( \tfrac{1}{24} \BNN_1^4 + \tfrac{1}{3} \BNN_1 \BNN_{1,2} \mright) \smod{p^4}. 
\end{align*}

The divided Bernoulli numbers $\BNN_\nu$ and $\BNN_{\nu,2}$ are embedded in the 
theory of the Kummer congruences. By this means, the results for $\WQ_p$ in 
different moduli can be easily reduced in a chain by
\[
  \WQ_p \smod{p^4} \;\longmapsto\; \WQ_p \smod{p^3} \;\longmapsto\;
  \WQ_p \smod{p^2} \;\longmapsto\; \WQ_p \smod{p}.
\]


\bibliographystyle{amsplain}

\end{document}